\documentclass[11pt]{amsart}
\usepackage{amssymb,amsmath,amsthm}
\usepackage[shortlabels]{enumitem}
\theoremstyle{plain}
\newtheorem{theorem}{Theorem}[section]
\newtheorem{corollary}[theorem]{Corollary}
\newtheorem{lemma}[theorem]{Lemma}

\theoremstyle{remark}
\newtheorem*{remark}{Remark}
\newtheorem*{remarks}{Remarks}

\newcommand{\CC}{{\mathbb C}}
\newcommand{\RR}{{\mathbb R}}
\newcommand{\ZZ}{{\mathbb Z}}
\renewcommand{\Re}{\operatorname{Re}}
\renewcommand{\Im}{\operatorname{Im}}

\DeclareMathOperator{\HS}{\rm HS}
\DeclareMathOperator{\sign}{\rm sign}

\begin{document}

\title[Real and imaginary parts of powers of the Volterra operator]{On the real and imaginary parts of powers of the Volterra operator}

\date{}

\author[T. Ransford]{Thomas Ransford}
\address{D\'epartement de math\'ematiques et de statistique, Universit\'e Laval,
Qu\'ebec City (Qu\'ebec),  G1V 0A6, Canada}
\email[Corresponding author]{ransford@mat.ulaval.ca}

\author[D. Tsedenbayar]{Dashdondog Tsedenbayar}
\address{Department of Applied Mathematics, Mongolian University of Science and Technology,
P.O.\ Box 46/520, Ulaanbaatar, Mongolia}
\email{cdnbr@yahoo.com}

\thanks{First author supported by grants from NSERC and the Canada Research Chairs program.}

\begin{abstract}
We study  the real and imaginary
parts of the powers of the Volterra operator on $L^2[0,1]$,
specifically their eigenvalues, their norms and their numerical ranges.
\end{abstract}


\makeatletter
\@namedef{subjclassname@2020}{\textup{2020} Mathematics Subject Classification}
\makeatother


\maketitle

%%%%%%%%%%%%%%%%%%%%%%%%%%%%%%%%%%%%%%%%%%%%%%%%

\section{Introduction}\label{S:intro}

Let $V: L^2[0,1]\to L^2[0,1]$ be the  Volterra operator, defined by
\[
Vf(x):=\int_0^x f(t)\,dt.
\]
It is well known that $V$ is a compact  quasi-nilpotent operator,
whose is adjoint given by
\[
V^*f(x)=\int_x^1 f(t)\,dt \quad(f\in L^2[0,1]).
\]

The powers $V^n$ have been the object of numerous articles over the years. A simple proof by induction shows that, for all $n\ge1$ and all $f\in L^2[0,1]$, we have
\begin{equation}\label{E:powers}
V^n f(x) =\int_0^x \frac{(x-t)^{n-1}}{(n-1)!}f(t)\,dt
\quad\text{and}\quad
V^{*n} f(x) =\int_x^1 \frac{(t-x)^{n-1}}{(n-1)!}f(t)\,dt.
\end{equation}

Rather less seems to be known about the real and imaginary parts of these powers,
namely 
$\Re V^n:=(V^n+V^{*n})/2$ and
$\Im V^n:=(V^n-V^{*n})/2i$.
It follows from \eqref{E:powers} that, for all $f\in L^2[0,1]$,
\begin{equation}\label{E:Repowers}
\Re V^n f(x)=\frac{1}{2}\int_0^1 \frac{|x-t|^{n-1}}{(n-1)!}f(t)\,dt
\end{equation}
and
\begin{equation}\label{E:Impowers}
\Im V^n f(x)=\frac{1}{2i}\int_0^1 \frac{|x-t|^{n-1}\sign(x-t)^{n-1}}{(n-1)!}f(t)\,dt.
\end{equation}
The operators $\Re V^n$ and $\Im V^n$ form the subject of the present work,
in which we study their eigenvalues, their norms and their numerical ranges.

%%%%%%%%%%%%%%%%%%%%%%%%%%%%%%%%%%%%%%%%%%%%%%%%

\section{Eigenvalues}\label{S:eigenvalues}

As remarked in the introduction, $V$ is a quasinilpotent operator, 
so $V^n$ can have no non-zero eigenvalues. In fact $V$ is injective, so
zero is not an eigenvalue of $V^n$ either.

The situation is quite different for $\Re V^n$ and $\Im V^n$. 
These are compact self-adjoint operators, so, 
as they are non-zero, they certainly have non-zero eigenvalues.
Moreover, since $V$ is a Hilbert--Schmidt operator, the
same is true for each of $\Re V^n$ and $\Im V^n$,
so their sets of eigenvalues are  square-summable sequences.
Our first result shows that, for certain values of $n$,
even more is true.

\begin{theorem}\label{T:finite}
If $n$ is odd, then $\Re V^n$ has at most $n$ non-zero eigenvalues.
If $n$ is even, then $\Im V^n$ has at most $n$ non-zero eigenvalues.
\end{theorem}

This result is an immediate consequence of the following lemma.

\begin{lemma}\label{L:finite}
Let $n\ge1$. 
The eigenvalues of $\Re (i^{n-1}V^n)$ are $0$, together with the  solutions $\lambda$ of
the equation $\det(A-\lambda B)=0$, 
where $A=(a_{jk})_{j,k=0}^{n-1}$ and $B=(b_{jk})_{j,k=0}^{n-1}$ are the $n\times n$ matrices given by
\[
a_{jk}:=\frac{i^{n-1}}{(j+n-k)!} 
\quad\text{and}\quad
b_{jk}:=
\begin{cases}
2/(j-k)!, & j\ge k,\\
0, &j<k.
\end{cases}
\]
\end{lemma}

\begin{proof}
Let $\lambda$ be a non-zero eigenvalue of $\Re (i^{n-1}V^n)$ with eigenvector $f$.
Then 
\begin{equation}\label{E:ev}
\frac{i^{n-1}}{2}(V^n+(-1)^{n-1}V^{*n})f=\lambda f.
\end{equation}
This expresses $f$ as the sum of integrals of $L^2$ functions,
so, changing $f$ on a set of measure zero, we may suppose that $f$ is continuous. 
Using \eqref{E:ev} to `bootstrap', we deduce that in fact $f\in C^\infty$.

Differentiating \eqref{E:ev} $n$ times, we obtain
\[
\lambda f^{(n)}=\frac{i^{n-1}}{2}(V^n f+(-1)^{n-1}V^{*n}f)^{(n)}=\frac{i^{n-1}}{2}(f+(-1)^{n-1}(-1)^nf)=0.
\]
As  $\lambda\ne0$, it follows that $f$ satisfies the differential equation
\begin{equation}\label{E:ode}
f^{(n)}=0.
\end{equation}
Also, differentiating \eqref{E:ev} $k$ times and then setting $x=1$, we obtain
the boundary conditions
\begin{equation}\label{E:bc}
i^{n-1}V^{n-k} f(1)=2\lambda f^{(k)}(1) \quad(k=0,1,\dots,n-1).
\end{equation}

The general solution of \eqref{E:ode} is 
\[
f(x)=\sum_{j=0}^{n-1} c_jx^j,
\]
where $c_0,c_1,\dots,c_{n-1}\in\mathbb{C}$.
Substituting this expression into \eqref{E:bc}, we obtain
\[
\sum_{j=0}^{n-1}\frac{i^{n-1}c_j}{(j+1)\cdots(j+n-k)}
=2\lambda \sum_{j=k}^{n-1} c_j j(j-1)\cdots(j-k+1) \quad(k=0,\dots,n-1).
\]
After simplification, this becomes
\[
\sum_{j=0}^{n-1}\frac{i^{n-1}}{(j+n-k)!}j!c_j=\lambda \sum_{j=k}^{n-1}\frac{2}{(j-k)!}j!c_j
\quad(k=0,\dots,n-1),
\]
in other words, $(A-\lambda B)v=0$, where $A,B$ are the matrices defined in the statement of the theorem,
and where $v$ is the vector with entries $v_j:=j!c_j~(j=0,\dots,n-1)$. The existence of a non-zero vector $v$
satisfying the equation $(A-\lambda B)v=0$ is equivalent to the condition that $(A-\lambda B)$ be a singular matrix,
 in other words, that $\det(A-\lambda B)=0$. 
We conclude that the non-zero eigenvalues of $\Re (i^{n-1}V^n)$ are precisely the solutions of the equation 
$\det(A-\lambda B)=0$.

Finally, we remark that, by the spectral theorem, a compact self-adjoint operator with only finitely many non-zero eigenvalues has a kernel of finite codimension. Thus $\Re (i^{n-1}V^n)$ has a kernel of finite codimension in $L^2[0,1]$.
In particular, the kernel is non-zero, and so zero is also an eigenvalue of $\Re (i^{n-1}V^n)$.
\end{proof}

For small values of $n$, Lemma~\ref{L:finite} permits us to calculate the eigenvalues explicitly. The results of these calculations are summarized in the following theorem.

\begin{theorem}\label{T:ev}
The non-zero eigenvalues of 
\begin{itemize}
\item $\Re V$ are $1/2$;
\item $\Im V^2$ are $\pm\sqrt{3}/12$;
\item $\Re V^3$ are $-1/24$ and $1/48\pm\sqrt{5}/80$;
\item $\Im V^4$ are $\pm\sqrt{1575\pm84\sqrt{345}}/5040$.
\end{itemize}
\end{theorem}

What about the remaining cases, namely $\Re V^n$ when $n$ is even and $\Im V^n$ when $n$ is odd?
Even when $n=1,2$, these cases turn out to be more complicated.

\begin{theorem}\label{T:ev1}
The eigenvalues of $\Im V$ are $\displaystyle\Bigl\{\frac{1}{(2n+1)\pi}: n\in\ZZ\Bigr\}$.
\end{theorem}

\begin{theorem}\label{T:ev2}
The eigenvalues of $\Re(V^2)$ are
\[
\Bigl\{\frac{1}{4t^2}:t>0,\,\frac{\coth t}{t}=1\Bigr\}\cup\Bigl\{-\frac{1}{4t^2}:t>0,\,\frac{\cot t}{t}=-1\Bigr\}\cup\Bigl\{-\frac{1}{(2n+1)^2\pi^2}:n\in\ZZ\Bigr\}.
\]
\end{theorem}

We carry out the proof just for Theorem~\ref{T:ev2}.  Theorem~\ref{T:ev1} is similar, but simpler.

\begin{proof}
Let $\lambda$ be an eigenvalue of $\Re V^2$ with eigenvector $f$.

We first show that $\lambda\ne0$.
Suppose, on the contrary, that $\frac{1}{2}(V^2+V^{*2})f=0$.
Since $Vf$ and $V^*f$ are continuous functions,
differentiation gives  $\frac{1}{2}(Vf-V^*f)=0$, in other words, $\int_0^x f=\int_x^1f$ for all $x\in [0,1]$.
This implies that $f=0$ a.e.\ on $[0,1]$, a contradiction. Thus $\lambda\ne0$, as claimed.

Since $\lambda$ is real and non-zero, it can be written as $\lambda=1/\omega^2$, where
$\omega$ is either real or pure imaginary. The eigenvalue condition then becomes
\begin{equation}\label{E:ev2}
f= \frac{\omega^2}{2}(V^2f+V^{*2}f).
\end{equation}
Using the bootstrap technique already mentioned in the proof of Theorem~\ref{T:ev}, 
we see that $f\in C^\infty$. 

Differentiating \eqref{E:ev2} twice leads to the differential equation
\begin{equation}\label{E:ode2}
f''=\omega^2 f.
\end{equation}
Also, differentiating \eqref{E:ev2} once, and setting $x=0,1$, and adding, 
we obtain the boundary condition
\begin{equation}\label{E:bc2a}
f'(0)+f'(1)=0.
\end{equation}
Lastly, setting $x=0,1$ in \eqref{E:ev2}, adding and simplifying, 
we obtain the second boundary condition
\begin{equation}\label{E:bc2b}
f(0)+f(1)=(1/2)(f'(1)-f'(0)).
\end{equation}

The general solution to the differential equation \eqref{E:ode2} can be written as 
\[
f(x)=A\Bigl(e^{\omega x}+e^{\omega(1-x)}\Bigr)+B\Bigl(e^{\omega x}-e^{\omega(1-x)}\Bigr),
\]
where $A,B$ are complex constants. Substituting this into the boundary condition \eqref{E:bc2a} gives
\[
2B(\omega+\omega e^{\omega})=0,
\]
and into \eqref{E:bc2b} yields
\[
2A(1+e^\omega)=-A(\omega-\omega e^\omega).
\]

Since $f\not\equiv 0$, at least one of $A,B$ is non-zero. If $B\ne0$, then $e^\omega=-1$, so $\omega=(2n+1)\pi i$
for some integer $n$. The corresponding  eigenvalue is
\[
\lambda=-\frac{1}{(2n+1)^2\pi^2}.
\]
If $A\ne0$, then $(1+e^\omega)=-(1/2)\omega(1-e^\omega)$, which can be rewritten as
\[
\frac{\coth(\omega/2)}{\omega/2}=1.
\]
If, further, $\omega$ is real, say $\omega=2t$, then
\[
\lambda=\frac{1}{4t^2} 
\quad\text{and}\quad 
\frac{\coth t}{t}=1.
\]
If, on the other hand, $\omega$ is purely imaginary, say $\omega=2it$, then
\[
\lambda=-\frac{1}{4t^2}
\quad\text{and}\quad 
\frac{\cot t}{t}=-1.
\]
This completes the list of eigenvalues of $\Re V^2$.
\end{proof}

We conclude this section with 
two general results about the distribution
of the eigenvalues of $\Re V^n$ and $\Im V^n$.

\begin{theorem}\label{T:sym}
For $n\ge1$, the eigenvalues of $\Im V^n$ are symmetrically distributed about $0$.
\end{theorem}

\begin{proof}
Let $U:L^2[0,1]\to L^2[0,1]$ be the unitary operator $Uf(x):=f(1-x)$.
A simple calculation shows that $V^*=U^{-1}VU$, whence $\Im V^n=-U^{-1}(\Im V^n) U$.
As $\Im V^n$ and $-\Im V^n$ are conjugate, they have the same eigenvalues.
This gives the result.
\end{proof}

The examples in Theorem~\ref{T:ev} and Theorem~\ref{T:ev2} show that there is no such symmetry result for $\Re V^n$.
Instead, we have the following theorem.

\begin{theorem}\label{T:Perron}
The eigenvalue of $\Re V^n$ of largest modulus is positive.
\end{theorem}

\begin{proof}
As remarked in \eqref{E:Repowers}, for $f\in L^2[0,1]$ we have
\[
\Re V^n f(x)=\frac{1}{2}\int_0^1 \frac{|x-t|^{n-1}}{(n-1)!}f(t)\,dt.
\]
In particular, $\Re V^n$ is an integral operator on $L^2[0,1]$
with non-negative
kernel. By the Krein--Rutnam theorem (see e.g.\ \cite[Theorem~7.10]{AA02}), 
the spectral radius of such an operator is an eigenvalue. 
The result follows.
\end{proof}

%%%%%%%%%%%%%%%%%%%%%%%%%%%%%

\section{Norms}\label{S:norms}

According to an old result of Halmos \cite[Problem~188]{Ha82},
the operator norm of $V$ is 
$\|V\|=2/\pi$. The values of $\|V^2\|$ and $\|V^3\|$ are more complicated, but they
can be expressed in terms of solutions to certain equations 
(see e.g.\ \cite[p.1058]{BD09}).
For larger values of $n$, one has to be content with 
numerical computations or inequalities.
It is known that, for all $n\ge1$,
 the operator norm $\|V^n\|$ and Hilbert--Schmidt
norm $\|V^n\|_{\HS}$ satisfy
\begin{equation}\label{E:normineq}
\frac{1}{(n-1)!}\frac{1}{\sqrt{(2n+1)(2n-1)}}\le \|V^n\|\le \|V^n\|_{\HS}=\frac{1}{(n-1)!}\frac{1}{\sqrt{2n(2n-1)}}.
\end{equation}
Here the right-hand equality has been known since \cite{LW97}.
The left-hand inequality was proved in \cite[Theorem~1.2]{BD09}.
The latter article also contains an account of the  history
of this topic.

Our goal in this section is to establish analogous result for $\Re V^n$ and $\Im V^n$. Since each of these is a compact self-adjoint operator,
its operator norm is simply the maximum modulus of its eigenvalues.
The results about eigenvalues established in \S\ref{S:eigenvalues} 
therefore lead immediately to the values listed in Table~\ref{Tb:norms}.

\begin{table}[ht]
\caption{Operator norms of $\Re V^n$ and $\Im V^n$
\newline(Here $\rho\approx1.199678640\ldots$ is the solution to $\coth\rho=\rho$.)}
\label{Tb:norms}
\renewcommand{\arraystretch}{1.2}
\begin{center}
\begin{tabular}{|c|c|c|}\hline
$n$ & $\|\Re V^n\|$ & $\|\Im V^n\|$\\
\hline 
$1$ & $1/2$ & $1/\pi$\\
$2$ & $1/(4\rho^2)$ & $\sqrt{3}/12$\\
$3$ & $1/48+\sqrt{5}/80$ &--\\
$4$ &--& $(\sqrt{1575+84\sqrt{345}})/5040$\\
\hline
\end{tabular}
\end{center}
\renewcommand{\arraystretch}{1}
\end{table}

It is possible to continue further, using Theorem~\ref{T:ev},
but only at the cost of obtaining just numerical values. 

We now turn to the problem of finding the analogues of the
inequality \eqref{E:normineq} for $\Re V^n$ and $\Im V^n$.
Our first result treats their Hilbert--Schmidt norms,
which we can evaluate exactly.

\begin{theorem}\label{T:HS}
For all $n\ge1$, we have
\[
\|\Re V^n\|_{\HS}=\|\Im V^n\|_{\HS}
=\frac{1}{\sqrt{2}(n-1)!}\frac{1}{\sqrt{2n(2n-1)}}.
\]
\end{theorem}

\begin{proof}
As already remarked in \eqref{E:Repowers} and \eqref{E:Impowers},
for $f\in L^2[0,1]$ we have
\[
\Re V^n f(x)=\frac{1}{2}\int_0^1 \frac{|x-t|^{n-1}}{(n-1)!}f(t)\,dt
\]
and
\[
\Im V^n f(x)=\frac{1}{2i}\int_0^1 \frac{|x-t|^{n-1}\sign(x-t)^{n-1}}{(n-1)!}f(t)\,dt.
\]
Since the Hilbert--Schmidt norm of an integral operator on $L^2[0,1]$ is  just the $L^2$-norm of its kernel on $[0,1]^2$, it follows that
\[
\|\Re V^n\|_{\HS}^2=\|\Im V^n\|_{\HS}^2=
\frac{1}{4(n-1)!^2}\int_0^1\int_0^1 (x-t)^{2n-2}\,dt\,dx.
\]
Using the binomial theorem, it is easy to check that
\begin{equation}\label{E:doubleint}
\int_0^1\int_0^1 (x-t)^{2n-2}\,dt\,dx=\frac{1}{n(2n-1)}.
\end{equation}
The result follows.
\end{proof}

\begin{remark}
Comparing the result of Theorem~\ref{T:HS} with \eqref{E:normineq},
we see that 
\begin{equation}\label{E:parallelogram}
\|\Re V^n\|_{\HS}=\|\Im V^n\|_{\HS}=\frac{\|V^n\|_{\HS}}{\sqrt{2}}.
\end{equation}
This could also be seen more directly, as follows. 
Since the Hilbert--Schmidt norm is a Hilbert-space norm, 
it satisfies the parallelogram identity, so
\[
\|V^n+V^{*n}\|_{\HS}^2+\|V^n-V^{*n}\|_{\HS}^2
=2\|V^n\|_{\HS}^2+2\|V^{*n}\|_{\HS}^2,
\]
in other words,
\[
2\|\Re V^n\|_{\HS}^2+2\|\Im V^n\|_{\HS}^2
=\|V^n\|_{\HS}^2+\|V^{*n}\|_{\HS}^2.
\]
Also, as the kernels of $\Re V^n$ and $\Im V^n$ have the same
absolute value a.e., we have $\|\Re V^n\|_{\HS}=\|\Im V^n\|_{\HS}$.
Lastly, since $V$ and $V^*$ are unitarily equivalent, we have $\|V^n\|_{\HS}=\|V^{*n}\|_{\HS}$.
The identity \eqref{E:parallelogram} follows immediately from these observations.
\end{remark}

Now we turn our attention to the operator norms 
$\|\Re V^n\|$ and $\|\Im V^n\|$.
Our goal is to establish the following theorem.

\begin{theorem}\label{T:opnorm}
For each $n\ge1$, we have
\[
\frac{1}{\sqrt{2}(n-1)!}\frac{1}{\sqrt{2n(2n+1)}}
\le \|\Re V^n\|\le 
\frac{1}{\sqrt{2}(n-1)!}\frac{1}{\sqrt{2n(2n-1)}}
\]
and
\[
\frac{1}{\sqrt{2}(n-1)!}\frac{1}{\sqrt{2n(2n+1)}}\Bigl(1-\binom{2n-2}{n-1}^{-1}\Bigr)
\le \|\Im V^n\|\le 
\frac{1}{\sqrt{2}(n-1)!}\frac{1}{\sqrt{2n(2n-1)}}.
\]
\end{theorem}

To prove this result, we adopt a similar strategy
to that used in \cite{BD09} to obtain  the bounds in
\eqref{E:normineq}. The idea is to estimate the operator
norm $\|S\|$ of a self-adjoint Hilbert--Schmidt operator $S$ using the
elementary general inequalities
\begin{equation}\label{E:genineq}
\frac{\|S^2\|_{\HS}}{\|S\|_{\HS}}\le\|S\|\le\|S\|_{\HS}.
\end{equation}

We have already evaluated 
$\|\Re V^n\|_{\HS}$ and $\|\Im V^n\|_{\HS}$. 
It remains to estimate
$\|(\Re V^n)^2\|_{\HS}$ and $\|(\Im V^n)^2\|_{\HS}$.
In view of \eqref{E:genineq}, we seek lower bounds.
We start with $\|(\Re V^n)^2\|_{\HS}$.

\begin{lemma}\label{L:sq}
For all $n\ge1$, we have
\[
\|(\Re V^n)^2\|_{\HS}
\ge \frac{1}{2(n-1)!^2}\frac{1}{(2n-1)\sqrt{2n(2n+1)}}.
\]
\end{lemma}

\begin{proof}
Using Fubini's theorem, we see that $(\Re V^n)^2$ is the
integral operator on $[0,1]$ with kernel
\[
K_n(x,y):=\frac{1}{4(n-1)!^2}\int_0^1 |x-t|^{n-1}|t-y|^{n-1}\,dt.
\]
The Hilbert--Schmidt norm of $(\Re V^n)^2$ is therefore the
$L^2$-norm of this kernel on $[0,1]^2$. Suppose that $0\le x\le y\le1$.
Then we have
\begin{align*}
\int_0^1|x-t|^{n-1}|t-y|^{n-1}\,dt
&\ge\Bigl(\int_0^x+\int_y^1\Bigr)|x-t|^{n-1}|t-y|^{n-1}\,dt\\
&= 2\int_0^x(x-t)^{n-1}(y-t)^{n-1}\,dt\\
&=2\int_0^1 (x-sx)^{n-1}(y-sx)^{n-1}x\,ds\\
&\ge2\int_0^1 (x-sx)^{n-1}(y-sy)^{n-1}x\,ds\\
&=\frac{2x^ny^{n-1}}{2n-1}.
\end{align*}
Hence
\begin{align*}
\|(\Re V^n)^2\|_{HS}^2
&=\int_0^1\int_0^1 K_n(x,y)^2\,dx\,dy\\
&=2\int_0^1\int_0^y K_n(x,y)^2\,dx\,dy\\
&\ge \frac{2}{16(n-1)!^4}\int_0^1\int_0^y\Bigl(\frac{2x^ny^{n-1}}{2n-1}\Bigr)^2\,dx\,dy\\
&=\frac{2}{16(n-1)!^4}\frac{4}{(2n+1)4n(2n-1)^2}\\
&=\frac{1}{4(n-1)!^4}\frac{1}{(2n+1)2n(2n-1)^2}.\qedhere
\end{align*}
\end{proof}

We cannot quite repeat this proof for $(\Im V^n)^2$,
since the corresponding kernel is given by an integral of
a function that is not positive everywhere. Nor can we assert
that $\|(\Re V^n)^2\|_{\HS}=\|(\Im V^n)^2\|_{\HS}$,
as was the case for the first powers of $\Re V^n$ and $\Im V^n$.
However, it turns out that $(\Re V^n)^2$ is extremely close
to $(\Im V^n)^2$ in the Hilbert--Schmidt norm, and this is good enough for our needs.

\begin{lemma}\label{L:sqdiff}
For all $n\ge1$, we have
\[
\|(\Re V^n)^2-(\Im V^n)^2\|_{\HS}
=\frac{1}{2(2n-1)!}\frac{1}{\sqrt{2n(4n-1)}}.
\]
\end{lemma}

\begin{proof}
Using elementary algebra, we have
\[
(\Re V^n)^2-(\Im V^n)^2=\frac{1}{4}(2V^{2n}+2V^{*2n})=\Re V^{2n}.
\]
The result therefore follows from Theorem~\ref{T:HS},
with $n$ replaced by $2n$.
\end{proof}

\begin{lemma}\label{L:sq2}
For all $n\ge1$, we have
\[
\|(\Im V^n)^2\|_{\HS}\ge 
\frac{1}{2(n-1)!^2}\frac{1}{(2n-1)\sqrt{2n(2n+1)}}
\Bigl(1-\binom{2n-2}{n-1}^{-1}\Bigr).
\]
\end{lemma}

\begin{proof}
By Lemmas~\ref{L:sq} and \ref{L:sqdiff} and the triangle inequality,
\begin{align*}
\|(\Im V^n)^2\|_{\HS}
&\ge \frac{1}{2(n-1)!^2}\frac{1}{(2n-1)\sqrt{2n(2n+1)}}
-\frac{1}{2(2n-1)!}\frac{1}{\sqrt{2n(4n-1)}}\\
&= \frac{1}{2(n-1)!^2}\frac{1}{(2n-1)\sqrt{2n(2n+1)}}
\Bigl(1-\frac{(n-1)!^2}{(2n-2)!}\frac{\sqrt{2n+1}}{\sqrt{4n-1}}\Bigr)\\
&\ge \frac{1}{2(n-1)!^2}\frac{1}{(2n-1)\sqrt{2n(2n+1)}}
\Bigl(1-\binom{2n-2}{n-1}^{-1}\Bigr).\qedhere
\end{align*}
\end{proof}

\begin{proof}[Proof of Theorem~\ref{T:opnorm}]
All that is left to do is to insert the information from Theorem~\ref{T:HS} and Lemmas~\ref{L:sq} and \ref{L:sq2}
into \eqref{E:genineq}. The result follows immediately.
\end{proof}

\begin{remarks}
(i) By Stirling's formula 
\[
\binom{2n-2}{n-1}\sim \frac{4^{n-1}}{\sqrt{\pi n}} \quad(n\to\infty).
\]
Thus although the estimate in Theorem~\ref{T:opnorm}
for $\|\Im V^n\|$ is slightly weaker than
that for $\|\Re V^n\|$, the difference for large $n$ is 
extremely small.

(ii) Comparing the result of Theorem~\ref{T:opnorm} with \eqref{E:normineq}, we see that
\[
\|\Re V^n\|\sim\|\Im V^n\|\sim \frac{\|V^n\|}{\sqrt{2}}\sim\frac{1}{(2\sqrt{2})n!} \quad(n\to\infty).
\]
Thus, although the exact analogue of \eqref{E:parallelogram}
fails to hold for the operator norm, it does hold asymptotically.
\end{remarks}

%%%%%%%%%%%%%%%%%%%%%%%%%%%%%%%%%

\section{Numerical ranges}\label{S:numranges}

It is known that the numerical range $W(V)$ of the Volterra 
operator is the convex compact set bounded by the vertical segment $[-i/2\pi,\,i/2\pi]$ and the curves
\[
t\mapsto \Bigl(\frac{1-\cos t}{t^2}\Bigr) \pm i\Bigl(\frac{t-\sin t}{t^2}\Bigr) \quad(t\in[0,2\pi]).
\]
This result is folklore. It appears in the middle of a discussion in \cite[p.113--114]{Ha82}, where it is attributed to A.~Brown. 

As for higher powers of $V$, to the best of our knowledge their
numerical ranges have never been identified precisely. However, the projections of $W(V^n)$ onto the $x$- and $y$-axes
are just $W(\Re V^n)$ and $W(\Im V^n)$ respectively
and, in principle at least, computing these is a much easier task. Indeed, 
each of $\Re V^n$ and $\Im V^n$  is a compact self-adjoint operator, so its numerical range is just the real interval obtained by taking the closed convex hull of the eigenvalues. 

For small values of $n$, the results in \S\ref{S:eigenvalues} lead to the intervals listed in Table~\ref{Tb:nr}. 

\begin{table}[ht]
\caption{Numerical ranges of $\Re V^n$ and $\Im V^n$
\newline(Here $\rho\approx1.199678640\ldots$ is the solution to $\coth\rho=\rho$.)}
\label{Tb:nr}
\renewcommand{\arraystretch}{1.8}
\begin{center}
\begin{tabular}{|c|c|c|}\hline
$n$ & $W(\Re V^n)$ & $W(\Im V^n)$\\
\hline 
$1$ & $\Bigl[0,\,1/2\Bigr]$ & $\Bigl[-1/\pi,\,1/\pi\Bigr]$\\
$2$ & $\Bigl[-1/\pi^2,\,1/(4\rho^2)\Bigr]$ & $\Bigl[-\sqrt{3}/12,\,\sqrt{3}/12\Bigr]$\\
$3$ & $\Bigl[-1/24,\,1/48+\sqrt{5}/80\Bigr]$ &--\\
$4$ &--& $\displaystyle\Bigl[-\frac{\sqrt{1575+84\sqrt{345}})}{5040},\,\frac{\sqrt{1575+84\sqrt{345}})}{5040}\Bigr]$\\
[0.2cm]\hline
\end{tabular}
\end{center}
\renewcommand{\arraystretch}{1}
\end{table}

For larger values of $n$, we have the following result.

\begin{theorem}\label{T:nr}
For $n\ge1$, we have
\[
\Bigl[-\frac{n-1}{3(n+3)}\frac{1}{(n+1)!},\,\|\Re V^n\|\Bigr]\subset W(\Re V^n)\subset \Bigl[-\|\Re V^n\|,\,\|\Re V^n\|\Bigr]
\]
and
\[
W(\Im V^n)= \Bigl[-\|\Im V^n\|,\,\|\Im V^n\|\Bigr].
\]
\end{theorem}

\begin{proof}
The inclusions 
\[
W(\Re V^n)\subset \Bigl[-\|\Re V^n\|,\,\|\Re V^n\|\Bigr]
\quad\text{and}\quad
W(\Im V^n)\subset \Bigl[-\|\Im V^n\|,\,\|\Im V^n\|\Bigr]
\]
are clear.
Moreover, the latter inclusion is actually an equality,
since, by Theorem~\ref{T:sym}, the eigenvalues of $\Im V^n$
are symmetrically distributed about the origin.

By Theorem~\ref{T:Perron}, $\|\Re V^n\|$ is an eigenvalue of $\Re V^n$.
Also, if $f(x):=1-2x$, then $\|f\|_2=1/\sqrt{3}$ and
\[
\langle\Re V^n f,f\rangle=\Re\langle V^nf,\, f\rangle=\int_0^1\Bigl(\frac{x^n}{n!}-2\frac{x^{n+1}}{(n+1)!}\Bigr)(1-2x)\,dx=-\frac{n-1}{n+3}\frac{1}{(n+1)!}.
\]
It follows that both
\[
\|\Re V^n\| \quad\text{and}\quad -\frac{n-1}{3(n+3)}\frac{1}{(n+1)!}
\]
belong to $W(\Re V^n)$, and therefore so do all the points in between, by convexity.
\end{proof}

A Hilbert-space operator $T$ is said to be \emph{accretive} if its numerical range is contained in the half-plane
$\{z\in\CC:\Re z\ge0\}$, or equivalently, if its real part is a positive operator.
Also $T$ is \emph{dissipative} if $-T$ is accretive.

Clearly the Volterra operator $V$ is accretive. This is no longer true of higher powers of $V$.

\begin{corollary}
If $n\ge2$, then $V^n$ is neither accretive nor dissipative.
\end{corollary}

\begin{proof}
As we have just shown in Theorem~\ref{T:nr}, 
if $n\ge2$ then $W(\Re V^n)$
contains both strictly positive and strictly negative values.
\end{proof}

A natural next step is to consider the accretivity of more general polynomials $p(V)$ of the Volterra operator. In principle, this amounts to determining whether the eigenvalues of $\Re p(V)$ are all non-negative. In practice this may be rather complicated. For quadratic polynomials, the calculations are similar to those for $\Re V^2$ already carried out in the proof of Theorem~\ref{T:ev2}. 

We record here one special case that is particularly simple, and for which eigenvalue calculations are not needed.

\begin{theorem}\label{T:accretive}
Let $a,b\in\RR$. Then $aV+bV^2$ is accretive if and only if 
$b\le 0$ and $2a+b\ge0$.
\end{theorem}

\begin{proof}
Note that $\Re V$ is the rank-one operator given by
\[
\Re Vf=\frac{1}{2}\Bigl(\int_0^1 f\Bigr)1 \quad(f\in L^2[0,1]).
\]
Hence, for all $f\in L^2[0,1]$, we have
\[
\Re \langle Vf,f\rangle=\langle \Re Vf,f\rangle=\frac{1}{2}\Big|\int_0^1 f\Bigr|^2=2\|(\Re V)f\|_2^2.
\]
Also
\begin{align*}
\Re \langle V^2f,f\rangle
&=\frac{1}{2}\Bigl(\langle Vf,V^*f\rangle+\langle V^*f,Vf\rangle\Bigr)\\
&=\frac{1}{4}\Bigl(\|(V+V^*)f\|_2^2-\|(V-V^*)f\|_2^2\Bigr)\\
&=\|(\Re V)f\|_2^2-\|(\Im V)f\|_2^2.
\end{align*}
Thus, if $a,b\in\RR$, then
\[
\Re \Bigl\langle (aV+bV^2)f,\,f\Bigr\rangle
=(2a+b)\|(\Re V)f\|_2^2-b\|(\Im V)f\|_2^2
\quad(f\in L^2[0,1]).
\]
In particular, $aV+bV^2$ is accretive if and only if
\begin{equation}\label{E:accretive}
(2a+b)\|(\Re V)f\|_2^2-b\|(\Im V)f\|_2^2\ge0
\quad(f\in L^2[0,1]).
\end{equation}

Taking $f(x):=(x-1/2)$, we have 
\[
\|(\Re V)f\|_2=\frac{1}{2}\Bigl|\int_0^1 f\Bigr|=0
\qquad\text{and}\qquad
\|(\Im V)f\|_2=(1/2)\|x^2-x\|_2>0.
\]
Substituting this information into \eqref{E:accretive}, 
we see that  \eqref{E:accretive} implies that 
$b\le0$. 

Also, taking $f_n:=n1_{[0,\,1/n]}+n 1_{[1-1/n,\,1]}$, we have
\[
\|(\Re V)f_n\|_2=\frac{1}{2}\Bigl|\int_0^1 f_n\Bigr|=1
\qquad\text{and}\qquad
\bigl\|(\Im V)f_n\bigr\|_2^2\le 2/n.
\]
(To see the last inequality, observe that $|(\Im V)f_n(x)|\le 1$ for all $x\in[0,1]$ and that
 $(\Im V)f_n(x)=0$ if $x\in[1/n,1-1/n]$.) Substituting these $f_n$ into \eqref{E:accretive}, and letting $n\to\infty$, we see that  \eqref{E:accretive} implies that 
$2a+b\ge0$. 

To summarize, we have shown that, if $aV^2+bV$ is accretive, then
we have $b\le0$ and $2a+b\ge0$. Conversely, if both these inequalities hold, then \eqref{E:accretive} clearly implies that $aV+bV^2$ is accretive.
\end{proof}

We end with a simple application of Theorem~\ref{T:accretive}.
It is a classic fact that $\|(I+aV)^{-1}\|\le 1$ for all $a\ge0$.
This is proved in \cite[p.302]{Ha82}  directly from the fact that $V$ is accretive. 
Theorem~\ref{T:accretive} thus permits us to deduce the following generalization.

\begin{corollary}
If $a\ge -b/2\ge0$, then $\|(I+aV+bV^2)^{-1}\|\le1$.
\end{corollary}

\bigskip

%\section*{Declarations}

%\subsection*{\textbf{Funding statement}}
%Ransford's research was supported by grants from the Natural Sciences and Engineering Research Council of Canada and from the Canada Research Chairs Program. Tsedenbayar has no relevant financial or non-financial interests to disclose.

%\subsection*{\textbf{Data availability statement}}
%Data sharing not applicable to this article as no datasets were generated or analyzed during the current study.

\bigskip

\bibliographystyle{amsplain}
\bibliography{biblist.bib}

\end{document}